\documentclass[a4paper,12pt]{article}
\usepackage{comment}
\usepackage{cite}
\usepackage{amsmath}
\usepackage{amssymb}
\usepackage{amsfonts}
\usepackage[T1]{fontenc}
\usepackage[utf8]{inputenc}
\usepackage{graphicx}
\usepackage{fancyhdr}
\usepackage{float}
\usepackage{xcolor}
\usepackage{authblk}
\usepackage{empheq}
\usepackage[hyphens]{url}
\usepackage{hyperref} 
\usepackage[]{breakurl}
\usepackage{mathrsfs}
\usepackage[]{amsthm}

\usepackage{geometry}
\newtheorem{theorem}{Theorem}

\begin{document}

\title{Explicit formulas for Euler's totient function and the number of divisors}

\author[$\dagger$]{Jean-Christophe {\sc Pain}$^{1,2,}$\footnote{jean-christophe.pain@cea.fr}\\
\small
$^1$CEA, DAM, DIF, F-91297 Arpajon, France\\
$^2$Universit\'e Paris-Saclay, CEA, Laboratoire Mati\`ere en Conditions Extr\^emes,\\ 
F-91680 Bruy\`eres-le-Ch\^atel, France
}

\date{}

\maketitle

\begin{abstract}
In this article, we present relations for the Euler totient function $\varphi(n)$ and the number of divisors $\tau(n)$ in terms of finite sums of integer parts of rational numbers or greatest common divisors of pairs of integers. Some of the formulas are obtained using a relation due to Menon and the connections with the Pillai arithmetic function are outlined. The reported expressions may be useful to derive new bounds for the usual arithmetical functions.
\end{abstract}

\section{Introduction}\label{sec1}

The number of divisors of natural number $n$ is
\begin{equation*}
    \tau(n)=\sum_{d|n}1,
\end{equation*}
and the Euler totient function is defined as
\begin{equation*}
    \varphi(n)=\sum_{\substack{k=1\\\mathrm{gcd}(k,n)=1}}^{n-1}1,
\end{equation*}
where $\mathrm{gcd}(k,n)$ denotes the greatest common divisor of $k$ and $n$. It is well known that
\begin{equation*}
    \varphi (n)=n\prod_{p\mid n}\left(1-{\frac {1}{p}}\right),
\end{equation*}
where $p$ are distinct prime numbers. The M\"obius inversion applied to the divisor sum formula gives (see for instance \cite{Apostol1976,Graham1989}):

\begin{equation*}
    \varphi (n)=\sum_{d\mid n}\mu \left(d\right)\, {\frac {n}{d}}=n\sum _{d\mid n}{\frac {\mu (d)}{d}},
\end{equation*}
where
\begin{equation*}
    \mu (n)=\begin{cases}
1&{\text{if }}n=1,\\
(-1)^{k}&{\text{if }}n{\text{ is the product of }}k{\text{ distinct primes}},\\
0&{\text{if }}n{\text{ is divisible by a square}}>1,
\end{cases}
\end{equation*}
is the M\"obius function. The Euler totient function satisfies
\begin{equation*}
    \varphi (mn)=\varphi (m)\,\varphi (n)\,{\frac {d}{\varphi (d)}}\quad {\text{where }}d=\operatorname {gcd} (m,n).
\end{equation*}
and it is worth mentioning the less known relation \cite{Dineva}:
\begin{equation*}
    \sum_{d\mid n}{\frac {\mu ^{2}(d)}{\varphi (d)}}={\frac {n}{\varphi (n)}}.
\end{equation*}

In section \ref{sec2}, the greatest common divisor is expressed in terms of integer parts. Based on this result, two expressions involving the Euler totient function $\varphi(n)$ are derived In section \ref{sec3}:
\begin{equation*}
    \varphi(n)=\frac{4}{n^2-3n+2}\,\sum_{j=1}^{n-1}\sum_{\substack{k=1\\\mathrm{gcd}(k,n)=1}}^n\left\lfloor\frac{jk}{n}\right\rfloor,
\end{equation*}
and
\begin{equation*}
    \varphi(n)=-\frac{n(n-1)}{2}+2\,\sum_{j=1}^{n-1}\sum_{k=1}^n\left\lfloor\frac{jk}{n}\right\rfloor\,\cos\left(\frac{2k\pi}{n}\right).
\end{equation*}
In section \ref{sec4}, a third identity, involving Euler's totient function as well as the number $\tau(n)$ of divisors of an integer $n$ (sometimes denoted by $d(n)$), is obtained using the Menon relation \cite{Menon1965,Toth2011,Zhao2017,Toth2018,Toth2019}. It reads
\begin{equation*}
    \varphi(n)=\frac{4}{2\,\tau(n)+n^2-5n+2}\,\sum_{j=1}^{n-1}\sum_{\substack{k=1\\\mathrm{gcd}(k,n)=1}}^n\left\lfloor\frac{j(k-1)}{n}\right\rfloor.
\end{equation*}
Finally, expressions of $\tau(n)$, are given in section \ref{sec5}, combining the above mentioned results, the Menon identity and the Pillai arithmetic function. They read respectively
\begin{equation*}
    \tau(n)=\frac{\displaystyle\sum_{\substack{k=1\\\mathrm{gcd}(k,n)=1}}^n\,\mathrm{gcd}(k-1,n)}{\sum_{k=1}^n\,\mathrm{gcd}(k,n)\,\cos\left(\frac{2\pi k}{n}\right)},
\end{equation*}
\begin{equation*}
    \tau(n)=\frac{(n^2-3n+2)\,\displaystyle\sum_{\substack{k=1\\\mathrm{gcd}(k,n)=1}}^n\,\mathrm{gcd}(k-1,n)}{4\,\displaystyle\sum_{j=1}^{n-1}\displaystyle\sum_{\substack{k=1\\\mathrm{gcd}(k,n)=1}}^n\left\lfloor\frac{jk}{n}\right\rfloor},
\end{equation*}
\begin{equation*}
    \tau(n)=\frac{\displaystyle\sum_{\substack{k=1\\\mathrm{gcd}(k,n)=1}}^n\,\mathrm{gcd}(k-1,n)}{2\left\{\displaystyle\sum_{j=1}^{n-1}\displaystyle\sum_{k=1}^n\left\lfloor\frac{jk}{n}\right\rfloor\,\cos\left(\frac{2k\pi}{n}\right)\right\}-\displaystyle\frac{n(n-1)}{2}},
\end{equation*}
and
\begin{equation*}
    \tau(n)=\frac{n^2-5n+2}{\left\{4\,\sum_{j=1}^{n-1}\displaystyle\sum_{\substack{k=1\\\mathrm{gcd}(k,n)=1}}^n\left\lfloor\frac{j(k-1)}{n}\right\rfloor-2\,\displaystyle\sum_{\substack{k=1\\\mathrm{gcd}(k,n)=1}}^n\,\mathrm{gcd}(k-1,n)\right\}}\sum_{\substack{k=1\\\mathrm{gcd}(k,n)=1}}^n\,\mathrm{gcd}(k-1,n).
\end{equation*}

\section{Expression of the greatest common divisor using integer parts}\label{sec2}

In this section we recall a known expression of the greatest common divisors of two natural numbers, that we will use in the following. Let us start with the following theorem \cite{Andreescu2009}:

\begin{theorem}

Let $a$, $b$, be nonnegative real numbers and $f: [a,b]\rightarrow [c,d]$ a bijective increasing function. One has
\begin{equation*}
    \sum_{a\leq k\leq b}\lfloor f(k)\rfloor+\sum_{c\leq k\leq d}\lfloor f^{-1}(k)\rfloor-n(\mathscr{G}_f)=\lfloor b\rfloor\,\lfloor d\rfloor-g(a)g(c),
\end{equation*}
where $k$ is integer, $n(\mathscr{G}_f)$ is the number of points with nonnegative integer coordinates on the graph of $f$ ang the function $g: \mathbb{R}\rightarrow\mathbb{Z}$ is defined by

\begin{equation*}
    g(x)=\begin{cases}
    \lfloor x\rfloor & \mathrm{ if } \;\;\;\; x\in\mathbb{R}\setminus\mathbb{Z}\\
    0 & \mathrm{ if } \;\;\;\; x=0\\
    x-1 & \mathrm{ if } \;\;\;\; x\in\mathbb{Z}\setminus \left\{0\right\}.
\end{cases}
\end{equation*}

\end{theorem}

\begin{proof}

For a bounded region $\mathscr{R}$ of the plane, we write $n(\mathscr{R})$ the number $n(\mathscr{R})$ of points in $\mathscr{R}$ with nonnegative integral coordinates. Let us consider the ensembles

\begin{align*}
\mathscr{R}_1=&\left\{(x,y)\in\mathbb{R}^2|\,a\leq x\leq b, 0\leq y\leq f(x)\right\},\\
\mathscr{R}_2=&\left\{(x,y)\in\mathbb{R}^2|\,c\leq y\leq d, 0\leq x\leq f^{-1}(y)\right\},\\
\mathscr{R}_3=&\left\{(x,y)\in\mathbb{R}^2|\,0\leq x\leq b, 0\leq y\leq d\right\},\\
\mathscr{R}_4=&\left\{(x,y)\in\mathbb{R}^2|\,0\leq x\leq a, 0\leq y\leq c\right\}.
\end{align*}
Then $n(\mathscr{R}_1)=\sum_{a\leq k\leq b}[f(k)]$, $n(\mathscr{R}_2)=\sum_{c\leq k\leq d}[f^{-1}(k)]$, $n(\mathscr{R}_2)=[b][d]$ and $n(\mathscr{R}_4)=g(a)\,g(c)$. We have
\begin{equation*}
    n(\mathscr{R}_1)+n(\mathscr{R}_2)-n(\mathscr{G}_f)=n(\mathscr{R}_3)-n(\mathscr{R}_4),
\end{equation*}
which completes the proof.

\end{proof}

\begin{theorem}
    
Let us denote $\mathrm{gcd}(m,n)$ the greatest common divisor of natural numbers $m$ and $n$. It can be shown that \cite{Andreescu2009,Landau1966}:
\begin{equation}\label{gcd1}
    \sum_{k=1}^{n}\left\lfloor\frac{km}{n}\right\rfloor+\sum_{k=1}^{m}\left\lfloor\frac{kn}{m}\right\rfloor=mn+\mathrm{gcd}(m,n).
\end{equation}
\end{theorem}

\begin{proof}

The list 
\begin{equation*}
    \frac{1.m}{n}, \frac{2.m_1}{n_1}, \cdots, \frac{n.m}{n}
\end{equation*}
contains exactly $\mathrm{gcd}(m,n)$ integers. Indeed, let $d$ be the greatest common divisor of $m$ and $n$, with $m=m_1.d$ and $n=n_1.d$. The list can be written
\begin{equation*}
    \frac{1.m_1}{n_1}, \frac{2.m_1}{n_1}, \cdots, \frac{n.m_1}{n_1}
\end{equation*}
and since $m_1$ and $n_1$ are relatively prime, there are $\lfloor\frac{n}{n_1}\rfloor$ integers among them.

Let us now consider the function $f: [1,n]\rightarrow\left[\frac{m}{n},m\right]$, such that $f(x)=m\,x/n$. We have $n(\mathscr{G}_f)=\mathrm{gcd}(m,n)$, which completes the proof. 

\end{proof}

Using the same kind of procedure, but considering the bijective decreasing function $f: [1,n]\rightarrow\left[0,m-\frac{m}{n}\right]$, such that $f(x)=-m\,x/n+m$, with $m\leq n$, we obtain a new relation which, combined to Eq. (\ref{gcd1}) yields (changing the letters for convenience):
\begin{equation}\label{gcd}
    \mathrm{gcd}(k,n)=2\sum_{j=1}^{n-1}\left\lfloor\frac{jk}{n}\right\rfloor+k+n-kn.
\end{equation}

\section{Two expressions of the Euler totient function}\label{sec3}

Summing both sides of Eq. (\ref{gcd}) over $k$, with the constraint that $k$ and $n$ are coprime (i.e., $\mathrm{gcd}(k,n)=1$) yields
\begin{equation*}
    \varphi(n)=2\sum_{\substack{k=1\\\mathrm{gcd}(k,n)=1}}^n\sum_{j=1}^{n-1}\left\lfloor\frac{jk}{n}\right\rfloor+(1-n)\sum_{\mathrm{gcd}(k,n)=1}^nk+n\,\varphi(n).
\end{equation*}
We have also
\begin{equation*}
    \sum_{\mathrm{gcd}(k,n)=1}^nk=\frac{1}{2}n\,\varphi(n)
\end{equation*}
and thus
\begin{equation}\label{res1}
    \varphi(n)=\frac{4}{n^2-3n+2}\,\sum_{j=1}^{n-1}\sum_{\substack{k=1\\\mathrm{gcd}(k,n)=1}}^n\left\lfloor\frac{jk}{n}\right\rfloor,
\end{equation}
which is the first main result of the present work.

We have also, since the totient is the discrete Fourier transform of the gcd, evaluated at 1 \cite{Schramm2008}:
\begin{equation}\label{fourier}
    \varphi(n)=\sum_{k=1}^n\,\mathrm{gcd}(k,n)\,\cos\left(\frac{2\pi k}{n}\right),
\end{equation}
and thus, multiplying Eq. (\ref{gcd}) by $\cos(2k\pi/n)$ and summing over $k$, one obtains
\begin{align*}
    \sum_{k=1}^n\,\mathrm{gcd}(k,n)\,\cos\left(\frac{2\pi k}{n}\right)=&2\sum_{k=1}^n\left\lfloor\frac{jk}{n}\right\rfloor\,\cos\left(\frac{2k\pi}{n}\right)\\
    &-(n-1)\,\sum_{k=1}^nk\,\cos\left(\frac{2k\pi}{n}\right)+n\,\sum_{k=1}^nk\cos\left(\frac{2k\pi}{n}\right).
\end{align*}
Now, since
\begin{equation*}
    \sum_{k=1}^nk\,\cos\left(\frac{2k\pi}{n}\right)=0
\end{equation*}
as well as
\begin{equation*}
    \sum_{k=1}^nk\,\cos\left(\frac{2k\pi}{n}\right)=\frac{n}{2},
\end{equation*}
one gets
\begin{equation}\label{res2}
    \varphi(n)=-\frac{n(n-1)}{2}+2\,\sum_{j=1}^{n-1}\sum_{k=1}^n\left\lfloor\frac{jk}{n}\right\rfloor\,\cos\left(\frac{2k\pi}{n}\right),
\end{equation}
which is the second main result of the present work.

\section{Third relation: expression of the Euler totient function from the Menon relation}\label{sec4}

The Menon relation (named after Puliyakot Kesava Menon, 1917-1979) reads \cite{Menon1965}
\begin{equation}\label{menon}
    \varphi(n)\,\tau(n)=\sum_{\substack{k=1\\\mathrm{gcd}(k,n)=1}}^n\,\mathrm{gcd}(k-1,n).
\end{equation}
If $\mathrm{gcd}(k,n)=1$, we have, according to Eq. (\ref{gcd}):

\begin{equation*}
    \mathrm{gcd}(k,n)=1\;\;\;\;\Rightarrow \;\;\;\; 2\sum_{j=1}^{n-1}\left\lfloor\frac{jk}{n}\right\rfloor+k+n-kn=1,
\end{equation*}
i.e.,
\begin{equation*}
    k+n-kn=1-2\sum_{j=1}^{n-1}\left\lfloor\frac{jk}{n}\right\rfloor.
\end{equation*}
We have also, still if $\mathrm{gcd}(k,n)=1$:
\begin{equation*}
\sum_{j=1}^{n-1}\left\lfloor\frac{jk}{n}\right\rfloor=\frac{(k-1)(n-1)}{2}
\end{equation*}
and using 

\begin{equation*}
    \mathrm{gcd}(k-1,n)=2\sum_{j=1}^{n-1}\left\lfloor\frac{j(k-1)}{n}{n}\right\rfloor+k-1+n-(k-1)n, 
\end{equation*}
the Menon relation yields
\begin{equation*}
    \varphi(n)\,\tau(n)=2\sum_{j=1}^{n-1}\sum_{k=1\\\substack{\mathrm{gcd}(k,n)=1}}^n\left\lfloor\frac{j(k-1)}{n}\right\rfloor-\frac{n(n-1)}{2}\,\varphi(n)+n\,\varphi(n)
\end{equation*}
and thus finally
\begin{equation}\label{res3}
    \varphi(n)=\frac{4}{2\,\tau(n)+n^2-5n+2}\,\sum_{j=1}^{n-1}\sum_{\substack{k=1\\\mathrm{gcd}(k,n)=1}}^n\left\lfloor\frac{j(k-1)}{n}\right\rfloor.
\end{equation}

\section{Relations for the number of divisors using the Pillai function}\label{sec5}

The Pillai arithmetical function, also referred to as the $\mathrm{gcd}$-sum function, is defined as \cite{Pillai1933,Broughan2002,Toth2010}:
\begin{equation*}
P(n)=\sum _{k=1}^{n}\mathrm{gcd}(k,n),
\end{equation*}
or equivalently
\begin{equation*}
P(n)=\sum _{d\mid n}d\,\varphi (n/d),
\end{equation*}
where $d$ is a divisor of $n$ and $\varphi$ is Euler's totient function. It also can be written as
\begin{equation*}
    P(n)=\sum _{d\mid n}d\,\tau (d)\,\mu (n/d),
\end{equation*}
where $\tau$ is the divisor function, and $\mu$ the M\"obius function (see section \ref{sec1}). Since one has
\begin{equation*}
    \varphi\left(\frac{n}{d}\right)=\frac{n}{d}\,\prod_{p|d}\left(1-\frac{1}{p}\right)
\end{equation*}
we can write
\begin{equation*}
    P(n)=n\,\prod_{p|n}\left[1+\nu_p(n)\,\left(1-\frac{1}{p}\right)\right],
\end{equation*}
where $\nu_p$ is the $p-$adic valuation of $n$. We have also
\begin{equation*}
    P(n)=\sum_{k=1}^n\sum_{d|gcd(k,n)}\varphi(d)=\sum_{k=1}^n\sum_{d|k}\sum_{d|n}\varphi(d)=\sum_{d|n}\varphi(d)\sum_{k=1}^n\sum_{d|k}1=n\sum_{d|n}\frac{\varphi(d)}{d}.
\end{equation*}

Combining the Menon identity (\ref{menon}) with Eq. (\ref{fourier}), one obtains
\begin{equation}\label{toto}
    \tau(n)=\frac{\displaystyle\sum_{\substack{k=1\\\mathrm{gcd}(k,n)=1}}^n\,\mathrm{gcd}(k-1,n)}{\displaystyle\sum_{k=1}^n\,\mathrm{gcd}(k,n)\,\cos\left(\frac{2\pi k}{n}\right)}.
\end{equation}
Using Eq. (\ref{res1}), one gets
\begin{equation*}
    \tau(n)=\frac{(n^2-3n+2)\,\displaystyle\sum_{\substack{k=1\\\mathrm{gcd}(k,n)=1}}^n\,\mathrm{gcd}(k-1,n)}{4\,\displaystyle\sum_{j=1}^{n-1}\displaystyle\sum_{\substack{k=1\\\mathrm{gcd}(k,n)=1}}^n\left\lfloor\frac{jk}{n}\right\rfloor}.
\end{equation*}
Similarly, Eq. (\ref{res2}) yields
\begin{equation*}
    \tau(n)=\frac{\displaystyle\sum_{\substack{k=1\\\mathrm{gcd}(k,n)=1}}^n\,\mathrm{gcd}(k-1,n)}{2\left\{\displaystyle\sum_{j=1}^{n-1}\displaystyle\sum_{k=1}^n\left\lfloor\frac{jk}{n}\right\rfloor\,\cos\left(\frac{2k\pi}{n}\right)\right\}-\displaystyle\frac{n(n-1)}{2}}.
\end{equation*}
Finally, Eq. (\ref{res3}) leads to
\begin{equation*}
    \tau(n)=\frac{2\,\tau(n)+n^2-5n+2}{4\,\sum_{j=1}^{n-1}\sum_{\substack{k=1\\\mathrm{gcd}(k,n)=1}}^n\left\lfloor\frac{j(k-1)}{n}\right\rfloor}\sum_{\substack{k=1\\\mathrm{gcd}(k,n)=1}}^n\,\mathrm{gcd}(k-1,n)
\end{equation*}
which gives, using Eq. (\ref{toto}):
\begin{equation*}
    \tau(n)=\frac{n^2-5n+2}{\left\{4\,\sum_{j=1}^{n-1}\sum_{\substack{k=1\\\mathrm{gcd}(k,n)=1}}^n\left\lfloor\frac{j(k-1)}{n}\right\rfloor-2\,\sum_{\substack{k=1\\\mathrm{gcd}(k,n)=1}}^n\,\mathrm{gcd}(k-1,n)\right\}}\sum_{\substack{k=1\\\mathrm{gcd}(k,n)=1}}^n\,\mathrm{gcd}(k-1,n).
\end{equation*}
Interesting formulas can also be obtained for the M\"obius function mentioned above, for instance using \cite{Hardy1980}:
\begin{equation*}
    \mu (n)=\sum _{\stackrel {1\leq k\leq n}{\gcd(k,\,n)=1}}e^{2\pi i{\frac {k}{n}}},
\end{equation*}
or
\begin{equation*}
    \sum _{k\leq n}\left\lfloor\frac {n}{k}\right\rfloor \mu (k)=1
\end{equation*}
as well as the more recent relation\cite{Kline2020}:
\begin{equation*}
    \sum _{jk\leq n}\sin\left(\frac{\pi jk}{2}\right)\mu (k)=1.
\end{equation*}

\section{Conclusion}\label{sec6}

We have discussed, using a relation due to Menon and with the help of the Pillai arithmetic function, expressions of the Euler totient function and the number of divisors $\tau(n)$ in terms of finite sums of integer parts of rational numbers or greatest common divisors. The formulas presented here may be of interest for the determination of bounds of the usual arithmetic functions \cite{Sandor2008,Sandor2015,Atanassov2001,Dimitrov2024}. We plan to try to derive new relations from extensions of the Pillai arithmetic function \cite{Sivaramakrishnan1971,Haukkanen2008,Sandor2001}, and investigate the case of Jordan's totient function ($k$ is a positive integer) \cite{Andrica2004}:
\begin{equation*}
    J_{k}(n)=n^{k}\prod_{p|n}\left(1-\frac {1}{p^{k}}\right),
\end{equation*}
where $p$ ranges through the prime divisors of $n$, and of the Mertens function:
\begin{equation*}
    M(n)=\sum _{k=1}^{n}\mu (k).
\end{equation*}
We also would like to take advantage of the link between the $\mathrm{gcd}$ and $q-$binomial coefficients \cite{Slavin2008}.

\end{document}